\documentclass{amsart}
\usepackage{latexsym}
\usepackage{amsmath,amssymb,amsfonts,amsthm,eufrak}
\usepackage{mathrsfs}
\usepackage{upref}
\usepackage{amscd}
\usepackage{eucal}
\usepackage[all]{xy}
\usepackage[dvips]{graphicx}

\begin{document}

\title[$H$-stability of Syzygy Bundles on Regular
Surfaces]{$H$-stability of Syzygy Bundles on some regular
Algebraic Surfaces }

\author{H. Torres-L\'opez}

\address{CONACyT - U. A. Matem\'aticas, U. Aut\'onoma de
Zacatecas
\newline  Calzada Solidaridad entronque Paseo a la
Bufa, \newline C.P. 98000, Zacatecas, Zac. M\'exico.}

\email{hugo@cimat.mx}

\author{A. G. Zamora}

\address{U. A. Matem\'aticas, U. Aut\'onoma de
Zacatecas
\newline  Calzada Solidaridad entronque Paseo a la
Bufa, \newline C.P. 98000, Zacatecas, Zac. M\'exico.}

\email{alexiszamora06@gmail.com}

\thanks{The first author was partially supported by project FORDECYT 265667. The second author was partially supported
by Conacyt Grant CB 2015-257079} \subjclass[2000]{14J60}
\keywords{Syzygy bundles, $H-$stability, Algebraic Surfaces}

\newtheorem{Theorem}{Theorem}[section]
\newtheorem{Lemma}[Theorem]{Lemma}
\newtheorem{Definition}[Theorem]{Definition}
\newtheorem{Proposition}[Theorem]{Proposition}
\newtheorem{Corollary}[Theorem]{Corollary}
\newtheorem{Remark}[Theorem]{Remark}

\newtheoremstyle{TheoremNum}
        {\topsep}{\topsep}              
        {\itshape}                      
        {}                              
        {\bfseries}                     
        {.}                             
        { }                             
        {\thmname{#1}\thmnote{ \bfseries #3}}
    \theoremstyle{TheoremNum}
    \newtheorem{rtheorem}{Theorem}
    \newtheorem{rlema}{Lemma}
\begin{abstract}

Let $L$ be a globally generated line bundle over a smooth
irreducible complex projective surface $X$. The syzygy bundle
$M_{L}$ is the kernel of the evaluation map $H^0(L)\otimes\mathcal
O_X\to L$. We prove the $L$-stability of $M_L$ for Hirzebruch
surfaces, del Pezzo surfaces and Enriques surfaces. The
$(-K_X)$-stability of syzygy bundles $M_L$ over del Pezzo surfaces
is also obtained.
\end{abstract}

\maketitle

\section{Introduction}

Let $X$ be a smooth irreducible projective variety over
$\mathbb{C}$ and let $L$ be a globally generated line bundle over
$X$ (from now on simply a generated bundle). The kernel $M_L$ of
the evaluation map $H^0(L)\otimes \mathcal{O}_X\rightarrow L$ fits
into the following exact sequence
\begin{equation}\label{dualspam}
 \xymatrix{ 0 \ar[r]^{} &  M_{L} \ar[r]^{}&  H^0(L)\otimes \mathcal{O}_X  \ar[r]^{} &  L \ar[r]^{} & 0.}
\end{equation}

\noindent The bundle $M_{L}$ is called a syzygy bundle.  The rank
of $M_{L}$ is $h^0(L)-1$. The vector bundles $M_L$ have been
extensively studied from different points of view.

When $X$ is a projective irreducible smooth curve of genus $g\geq
1$ L. Ein and R. Lazarsfeld  showed  in \cite{ein} that the syzygy
bundle $M_L$ is stable for $d>2g$ and it is semi-stable for $d=2g$
(see also \cite{butler}). After this, the semi-stability of $M_L$
was proved for line bundles with $\text{deg}(L)\geq
2g-\text{Cliff}(C)$ (see \cite[Corollary 5.4]{mistretastopino} and
\cite[Theorem 1.3]{camerecurvas}). In (\cite{ram}), Paranjape and
Ramanan proved that $M_{K_C}$ is semi-stable and is stable if $C$
is non-hyperelliptic. In \cite{schneider},  Schneider showed that
$M_L$ is semi-stable for a general curve $C$ (see also
\cite{but}). The semi-stability for incomplete linear series  over
general curves  was proved in \cite{leticiaynewstead}.

\vspace{.1cm}

In \cite{flenner}, Flenner showed the stability of $M_L$ for
projective spaces. The stability of syzygy bundles  for incomplete
linear series in projective spaces  has been studied by several
authors (see \cite{coanda}, \cite{rosa}, \cite{brenner}).
Recently, in \cite{CaucciLahoz2020}, the authors proved that given
an Abelian variety $A$ and any ample line bundle $L$ on $A$ the
syzygy bundle $M_{L^d}$ is $L$-stable if $d\ge 2$.

In the case of a projective surface $X$, we must start by fixing a
polarization $H$ and then  ask for the $H$-stability of the syzygy
bundle $M_L$. Recall that $H$-stability for a vector bundle $E$ on
$X$ means that for any sub-bundle $F\subset E$

$$\mu_H(F):= \frac{c_1(F).H}{\text{rk}(F)} < \mu_H(E):=
\frac{c_1(E).H}{\text{rk}(E)}.$$

In the pioneer work of Camere (\cite{camere}) it was proved that
$M_L$ is $L$-stable for any ample and generated $L$ on a $K3$
surface and for any generated $L$ with $L^2\ge 14$ on an abelian
surface $X$. In \cite{lazarsfeldsurfaces} Ein, Lazarsfeld and
Mustopa fixed an ample divisor $L$ and an arbitrary divisor $D$
over $X$, and setting $L_d=dL+D$ ($d\in \mathbb{N}$) showed that
$M_{L_d}$ is $L$-stable for $d>>0$.

Most of our results derive from the following: \vskip0.1cm

\textbf{Theorem 2.2} {\it Let $X$ be a smooth projective surface.
Let $L$ be an ample and generated line bundle over $X$ and $H$
 be a divisor such that an irreducible and non-singular curve  $C$ exists in $|H|$. Assume that
\begin{enumerate}
 \item $h^1(L-H)=0$;
 \item $h^0(H)\geq h^0(L\vert_C)$;
 \item $M_{L\vert_C}$ is semi-stable;
 \end{enumerate}
 Then $M_L$ is  $H$-stable.}

\vskip0.1cm

Section 2 is devoted to the proof of Theorem \ref{Theorem2}.

If $L=H$ and $X$ is regular, then conditions (1) and (2) are
automatically satisfied and in order to prove the stability of
$M_L$ it is sufficient to prove  the semi-stability of
$M_{L\vert_C}$. In section 3, using this idea, we obtain the
$L$-stability of the syzygy bundle $M_L$ in the following cases:
if $\vert L \vert$ contains either a genus $\le 1$ curve or a
Brill--Noether general curve (Corollary \ref{Cor1}), if $X$ is
either a del Pezzo (Corollary \ref{delPezzo}) or a Hirzebruch
surfaces (Corollary \ref{F_n}) and if $X$ is an  Enriques surfaces
under the condition that $\text{Cliff}(C)\geq 2$ (Corollary
\ref{Enriques}).

Moreover, we study the stability of syzygy bundle over del Pezzo
surfaces with respect to  the anti-canonical polarization:

\vskip0.1cm

\textbf{Theorem 3.7} {\it Let $X$ be a del Pezzo surface and let
$L$ be a generated line bundle on $X$. If $L$ contains an
irreducible curve, then the vector bundle $M_L$ is
$(-K_X)$-stable.}

\vskip0.1cm


\vskip0.1cm

\textbf{Conventions:} We work over the field of complex numbers
$\mathbb{C}$.  Given a coherent sheaf $\mathcal{G}$ on a variety
$X$ we write $h^i(\mathcal{G})$ to denote the dimension of the
$i$-th cohomology group $H^i(X,\mathcal{G})$.  The sheaf $K_X$
will denote the canonical sheaf on $X$. A surface always means a
smooth (or non-singular) projective complex irreducible surface.

\section{Stability on some regular surfaces}

The aim of this section is to prove Theorem \ref{Theorem2}.

Let $X$ be a projective, irreducible and non-singular surface $X$.
Let $H$ be an ample line bundle on $X$. The $H$-slope of a vector
bundle  $E$ is given by
\begin{eqnarray*}
\mu_H(E):=\frac{c_1(E).H}{\text{rk}(E)},
\end{eqnarray*}
where $c_1(E)$ is the first Chern class of $E$ and $\text{rk}(E)$
is the rank of $E$. In particular, the $H$-slope of $M_L$ is given
by
\begin{eqnarray*}
\mu_H(M_L)=\frac{c_1(M_L).H}{\text{rk}(M_L)}=\frac{-L.H}{h^0(L)-1}.
\end{eqnarray*}
Sometimes we write $\mu(E)$ by $\mu_H(E)$, when there is not
confusion about the choice of $H$. Note that if $C\in \vert H
\vert $ is non-singular projective and irreducible, we can compute
the $H$-slope of a vector bundle $E$ as
\begin{eqnarray*}
\mu_H(E)=\mu(E\vert_C):=\frac{\text{deg}(E\vert_C)}{\text{rk}(E)},
\end{eqnarray*}
where $\text{deg}(E\vert_C)$ is the degree of vector bundle $E$
restricted to the curve $C$.

Assume that there exists  an irreducible and non-singular curve
$C$ in the linear system $\vert H \vert$ and take a point $x\in C$. For
the remainder of the argument $H$, $C \in \vert H \vert$ and $x\in C$ would be fixed. Note that
$C\in \vert H\otimes m_x\vert$.

 Given any sub-bundle
$F\subset M_{L}$ we have, thanks to the condition
$h^1(L-H)=0$, a commutative diagram:

\begin{equation} \label{eq:ES}
\xymatrix{0  \ar[r]  & K \ar[r] \ar@{^{(}->}[d]      &  F|_C \ar[r] \ar@{^{(}->}[d] & N \ar[r] \ar@{^{(}->}[d]^{}  & 0\\
                0\ar[r]^{}  & H^0(L-H)\otimes \mathcal{O}_C \ar[r]_{}& (M_{L})|_C \ar[r]_{}   & M_{L\vert_C} \ar[r]^{} & 0    ,\\
                } \end{equation}
(see \cite{lazarsfeldsurfaces} page 76 and also \cite{camere},
proof of Proposition 3).

\begin{Lemma}\label{MainTheorem} Let $X$ be a smooth projective surface. Let $L$ be an ample and generated line bundle over $X$ and
$H$ be a divisor such that an irreducible and non-singular curve
$C$ exists in $\vert H \vert$.  Assume that the following
statements are satisfied:
\begin{enumerate}
 \item $h^1(L-H)=0$.
 \item $M_{L\vert_C}$ is semi-stable. \end{enumerate}

 Let $F\subset M_L$ be a sub-bundle with $\mu_H(F)\geq \mu_H(M_L)$.
Then,

$$ \text{rk}(F)\geq h^0(H)-1 + \text{rk}(K).$$

\end{Lemma}

The proof of the Lemma is quite analogous to the one of Lemma 1.1
in \cite{lazarsfeldsurfaces}. We only need to replace $B$ by $L-H$
and $L_d-B$ by $H$ and to compute the dimension of a fiber. We
include a proof for the reader convenience.

\begin{proof}

Note that for any vector bundle $E$ on $X$,
$\mu_H(E)=\mu(E\vert_C)$. Therefore the condition (3) is equivalent to $\mu(F\vert_C)\ge \mu
((M_{L})\vert_C)$. If $K=0$ in the exact sequence \eqref{eq:ES}, then
\begin{eqnarray*}
\mu(N)=\mu(F\vert_C)\geq
\mu((M_{L})\vert_C)>\mu(M_{L\vert_C})
\end{eqnarray*}
gives a contradiction with the semi-stability of $M_{L\vert_C}$
and thus we have  $K\neq 0$.

The multiplication map of sections
\begin{eqnarray*}
\nu:  \mathbb{P}(H^0(H\otimes m_x)) \times
\mathbb{P}(H^0(L-H))\rightarrow \mathbb{P}(H^0(L\otimes m_x))
\end{eqnarray*}
is a  finite morphism. After localizing at the given point $x$ we
obtain a commutative diagram:
\begin{equation*}
\xymatrix{  & {(K)}_x  \ar@{^{(}->}[r] \ar@{^{(}->}[d]      &  (F|_C)_{x}  \ar@{^{(}->}[d] &   & \\
                    & H^0(L-H) \ar@{^{(}->}[r]& {((M_{L})\vert_C)}_x=H^0(L\otimes m_x).    &  &    ,\\  }
\end{equation*}

Let $Z= \nu^{-1}(\mathbb{P} (F|_C)_x)$, since $\nu$ is finite
$\dim \mathbb{P}((F|_C)_{x}) \ge \dim Z$. Given $s\in H^0(
H\otimes m_x)$, we have that $s$ induces the injective morphism
\begin{eqnarray*}
H^0(L-H) &\rightarrow &  H^0(L\otimes m_x ) \\
\phi &\mapsto & s\otimes \phi.
\end{eqnarray*}
Therefore for any $\phi$ in the image of the morphism $(K)_x \to
H^0(L-H)$, we have that $(s,\phi)\in Z$ and $\pi_2(s,\phi)=s$. It
follows  that the projection
\begin{eqnarray*}
\pi_2:Z\rightarrow \mathbb{P}(H^0(H\otimes m_x))
\end{eqnarray*}
is dominant and the dimension of the general fibre is greater or
equal than $\text{rk}(K)-1$. Hence
\begin{eqnarray}\label{cotaF}
\text{rk}(F)\geq h^0(H)+\text{rk}(K)-1.
\end{eqnarray}

\end{proof}

\begin{Theorem}\label{Theorem2}
 Let $X$ be a smooth projective complex surface. Let $L$ be an ample and generated line bundle over $X$ and $H$ be a divisor such that
 an irreducible and non-singular curve $C$ exists in $\vert H
 \vert$. Assume that:
\begin{enumerate}
 \item $h^1(L-H)=0$,
 \item $h^0(H)\geq h^0(L\vert_C)$,
 \item $M_{L\vert_C}$ is semi-stable.
 \end{enumerate}
 Then $M_L$ is $H$-stable.
 \end{Theorem}

 \begin{proof}
Assume that $M_L$ is $H$-unstable, let $F\subset M_L$ be a sub-bundle
such that $\mu_H(F)\geq \mu_H(M_L)$. By Lemma \ref{MainTheorem}, we have
\begin{eqnarray*}
\text{rk}(F)\geq h^0(H)+\text{rk}(K)-1.
\end{eqnarray*}
We can repeat the proof of Lemma 1.2 in \cite{lazarsfeldsurfaces}
replacing $B$ by $L-H$ and $L_d-B$ by $H$. More specifically, in
Lemma 1.2 in \cite{lazarsfeldsurfaces} an inequality (inequality
(*) at the beginning of page 78) is obtained that translated to
our situation, just as we did with Lemma \ref{MainTheorem},
yields:

 $$\text{rk}(K)\ge h^0(L-H)\cdot\left(\frac{\text{rk} (F)}{\text{rk}(M_L)}\right).$$

 Therefore,
 we have:
\begin{eqnarray*}
\text{rk}(F)&\geq& h^0(H)+\text{rk}(K)-1\\
&\geq & h^0(H)+ h^0(L-H)\cdot
\left(\frac{\text{rk}(F)}{\text{rk}(M_{L})}\right)-1.
\end{eqnarray*}
That is,
\begin{eqnarray*}
\text{rk}(F)\geq
\frac{(h^0(H)-1)\text{rk}(M_L)}{\text{rk}(M_L)-h^0(L-H)}.
\end{eqnarray*}
Since $\text{rk}(M_L)-h^0(L-H)=\text{rk}(M_L\vert_C)=h^0(L\vert_C)-1$,
from the hypothesis $(2)$, we obtain that $\text{rk}(F)\geq
\text{rk}(M_L)$ which is impossible because $F$ is a sub-bundle of
$M_L$. This proves the theorem.
 \end{proof}

\section{Applications to regular surfaces}

Theorem \ref{Theorem2} is especially useful when $X$ is a regular
surface and $L=H$. In this case conditions $(1)$ and $(2)$ are
automatically satisfied and only $(3)$ remains to be verified.
Also, as $H=L$ is ample and generated we can assume that
$h^0(L)\ge 3$. Thus, a curve $C\in \vert L \vert$ exists which is
non-singular and connected, and therefore irreducible (Bertini's
Theorem).

\begin{Corollary}\label{Cor1} Assume $X$ is regular and $L$ is an ample
and generated line bundle on $X$. Let $C\in \vert L \vert$ be
irreducible and non-singular and assume that either:
\begin{enumerate}
\item  $g(C)\le 1$ or
\item $C$ is Brill--Noether general,
\end{enumerate}
then $M_L$ is $L$-stable. \end{Corollary}

\begin{proof}
Assume that $g(C)=0$. First note that  $h^1(L\vert_C)=0$ because
$L|_C$ is a globally generated  bundle over $C$. Taking cohomology
in the exact sequence:
$$0\to M_{L\vert_C} \to H^0(L\vert_C)\otimes \mathcal{O}_C \to L\vert_C \to
0,$$  we see that $h^0( M_{L\vert_C})=h^1( M_{L\vert_C})=0$. By
Grothendieck's Theorem, we get $M_{L\vert_C} \simeq \oplus
\mathcal{O}_C(-1)$. Therefore, $M_{L\vert_C}$ is semi-stable and
the result follows at once from Theorem \ref{Theorem2}.

Next, let $C$ be an elliptic non-singular curve  and let $L$ be a
line bundle on $C$ of degree $d\geq 2$, observe that $M_L$ is
stable. Indeed, the rank of $M_L^{\vee}$ is $d-1$, therefore the
slope of $M_L^{\vee}$ is given by
\begin{eqnarray*}
\mu(M_L^{\vee})=\frac{d}{d-1}.
\end{eqnarray*}

Let $F$ be a quotient of $M_L^{\vee}$, we want to prove that
$\mu(M_L^{\vee})<\mu(F).$ Since $F^{\vee}$ is a sub-bundle of
$M_L$, we get that $H^0(F^{\vee})=0$ and thus $H^1(F)=0$, by Serre
duality. By Riemann Roch Theorem, we get
\begin{eqnarray*}
h^0(F)= \text{deg}(F) + \text{rk}(F)(1-g(C))+h^1(F)=\text{deg}(F).
\end{eqnarray*}

Since $F$ is globally generated over $C$, it follows that
$h^0(F)\geq \text{rk}(F)$. Assume that $h^0(F)=\text{rk}(F)$, then
the evaluation map
 \begin{eqnarray*}
 H^0(F)\otimes \mathcal{O}_C \to F
 \end{eqnarray*}
is an isomorphism, which is impossible because $H^1(F)=0.$
Therefore $\text{deg}(F)\geq \text{rk}(F)+1$ and
\begin{eqnarray*}
\mu(F)\geq
\frac{\text{rk}(F)+1}{\text{rk}(F)}>\frac{d}{d-1}=\mu(M_L^{\vee}),
\end{eqnarray*}
the last inequality following from $\text{rk}(F)<d-1.$

Finally, assume that $C$ is Brill--Noether general. By
\cite{schneider} we have that $M_{L\vert_C}$ is semi-stable.
\end{proof}

Another general case that can be treated is when the anticanonical
divisor $-K_X$ is nef.

\begin{Proposition}\label{qigual0}
Let $X$ be a regular surface and let  $L$ be an ample and
generated line bundle on $X$. Let $C\in |L|$ be irreducible and
non-singular. Then:
\begin{enumerate}

   \item if $-L.K_X\ge 2$, then $M_L$ is $L$-stable.
    \item if $-L.K_X= 1$ and $\text{Cliff}(C)\geq 1$; then $M_L$ is $L$-stable.
    \item if $L.K_X= 0$ and $\text{Cliff}(C)\geq 2$; then $M_L$ is $L$-stable.
\end{enumerate}
\end{Proposition}

\begin{proof}

If $g(C)\le 1$, then apply Corollary \ref{Cor1}. If $-L.K_X\geq
2$, then by the Adjunction Formula we have $\deg (L\vert_C)=L^2\ge
2g(C)$. Therefore $M_{L\vert_C}$ is semi-stable and $M_L$ is
$L$-stable by Theorem \ref{Theorem2}. This proves (1). The proofs
of (2) and (3) are similar: using Adjunction Formula we get $\deg
(L\vert_C)=L^2\geq 2g(C)-\text{Cliff}(C),$ therefore $M_L\vert_C$
is semi-stable.

\end{proof}

Some particular cases of this situation are:

\begin{Corollary}\label{delPezzo}
Let $L$ be an ample and generated line bundle over a del Pezzo
surface, then $M_L$ is $L$-stable.
\end{Corollary}

\begin{proof} If $C\in \vert L \vert$ is irreducible and
non-singular and $g(C)\le 1$, then the result follows from
Corollary \ref{Cor1}. Otherwise, from $h^0(-K_X)\ge 2$ we have
$-L.K_X\ge 2$ and the result follows from Proposition
\ref{qigual0}. \end{proof}

\begin{Corollary}\label{F_n}
Let $\mathbb{F}_n$ be a non-singular Hirzebruch surface and let
$L$ be an ample and generated line bundle. Then $M_L$ is
$L$-stable.
\end{Corollary}
\begin{proof}
The canonical line bundle is given by $K_{\mathbb{F}_n}=
-2C_n-(n+2)F,$ where $C_n$ and $F$ are respectively the section
and the fiber of the structural fibration $\mathbb{F}_n \to
\mathbb{P}^1$. If $L$ is ample, then $c_1(L)=aC_n+bF$ with $a>0$
and $b>na.$ Since $-L.K_{\mathbb{F}_n}=2b-an+2a>3$, it follows
that $M_L$ is $L$-stable by Proposition \ref{qigual0}.
\end{proof}

The $L$-stability of $M_L$ for $X$ a $K3$ surface was studied in
(\cite{camere}, Theorem 1). As a byproduct of our method we can
recover the quoted result:

\begin{Corollary}
 Let $X$ be a smooth projective $K3$ surface and $L$ an ample and generated line bundle over $X$.
 Then $M_L$ is $L$-stable.
\end{Corollary}
\begin{proof} Let $C\in \vert L \vert$ be irreducible and
non-singular. By Adjunction Formula, $L\vert_C=K_C$ is the
canonical line bundle. In (\cite{ram}), Paranjape and Ramanan
proved that $M_{K_C}$ is semi-stable. The Corollary follows from
Theorem \ref{Theorem2}.
\end{proof}

Similar results can be obtained for regular surfaces with
numerically trivial canonical divisor:

\begin{Corollary}\label{Enriques}
Let $X$ be an Enriques surface.  Let $L$ be an ample and generated
line bundle on $X$. Assume that an irreducible and non-singular
curve $C$ in the linear system $|L|$ exists such that
$\text{Cliff}(C)\geq 2$. Then $M_L$ is $L$-stable.
\end{Corollary}

\begin{proof}
Note that $\text{deg}(L|_C)=L^2=2g-2\geq 2g-\text{Cliff}(C)$, thus
$M_{L\vert_C}$ is semi-stable.
\end{proof}

\vspace{.3cm}

Finally we address the case of $(-K_X)$-stability for del Pezzo
surfaces.

\begin{Theorem}\label{-KdelPezzo}
Let $X$ be a del Pezzo surface and let $L$ be a globally generated
line bundle on $X$. If $\vert L \vert$ contains an irreducible
curve, then the vector bundle $M_L$ is $(-K_X)$-stable.
\end{Theorem}

\begin{proof}
Note that $h^1(L)=h^2(L)=0.$ From Riemann-Roch Theorem we get
\begin{eqnarray*}
h^0(L)=1+ \frac{L^2-L.K_X}{2}.
\end{eqnarray*}
Therefore the rank of $M_L$ is equal to $\frac{L^2-L.K_X}{2}$. We
want to prove that $M_L^{\vee}$ is stable with respect to the
polarization $-K_X$. Let $C$ be a non-singular projective curve in
the linear system $|-K_X|.$ By Adjunction Formula $C$ is an
elliptic curve. The slope of $M_L^{\vee}$ with respect to $-K_X$
is given by
\begin{eqnarray*}
\mu(M_L^{\vee})=-\frac{2L.K_X}{L^2-L.K_X}.
\end{eqnarray*}
Let $F$ be a torsion-free quotient sheaf of $M_L^{\vee}$ of rank
$0<\text{rk} F<r;$ then $F|_C$ is a quotient  of
$(M_L^{\vee})|_C$. We want to prove that
\begin{eqnarray*}
\mu(M_L^{\vee})<\mu(F).
\end{eqnarray*}
First we assume that $F|_C$  is a vector bundle  on $C$. The
following properties are satisfied:
    \begin{enumerate}
    \item[(i)] $H^1(C,F|_C)=0.$
    \item[(ii)] $\text{deg}(F|_C)\geq \text{rk}(F)+1.$
    \end{enumerate}

Indeed, since $F$ is globally generated, we have $F|_C$ is
globally generated, and therefore $h^0(F^{\vee}|_C)=0$  and
$h^1(F|_C)=0$ by Serre duality, this proves (i).  To prove (ii),
by Riemann Roch Theorem for curves, we get
\begin{eqnarray*}
h^0(F|_C)= \text{deg}(F|_C) +
\text{rk}(F|_C)(1-g(C))+h^1(F|_C)=\text{deg}(F|_C).
\end{eqnarray*}
 Since $F|_C$ is globally generated over $C$, it follows that $h^0(F|_C)\geq \text{rk}(F|_C)$. Assume that $h^0(F|_C)=\text{rk}(F|_C)$,
 then the evaluation map
 \begin{eqnarray*}
 H^0(F|_C)\otimes \mathcal{O}_C \to F|_C
 \end{eqnarray*}
is an isomorphism but this is impossible because $H^1(F|_C)=0.$
Therefore $\text{deg}(F|_C)\geq \text{rk}(F)+1$ and we obtain
$(ii).$

Hence, we have
\begin{eqnarray*}
\mu(F)=\mu(F|_C)\geq
1+\frac{1}{\text{rk}(F|_C)}>1+\frac{2}{L^2-L.K_X}
\end{eqnarray*}
therefore
\begin{eqnarray*}
\mu(M_L^{\vee})= -\frac{2L.K_X}{L^2-L.K_X} \leq
1+\frac{2}{L^2-L.K_X} < \mu(F),
\end{eqnarray*}
the first inequality follows by Adjunction Formula and the fact
that $\vert L \vert$ contains an irreducible curve. Hence
$\mu(M_L^{\vee})<\mu(F)$.

Now, if $F|_C$ is not a vector bundle, then $F|_C=E\oplus \tau$,
where $E$ is a vector bundle and $\tau$ is a torsion sheaf over
$C$. The above proof can be repeated to obtain
$\mu(M_L^{\vee})<\mu(E)$. Then
\begin{eqnarray*}
\mu(M_L^{\vee})<\mu(E)\leq \mu(E\oplus \tau)=\mu(F)
\end{eqnarray*}
Therefore $M_L^{\vee}$ is stable with respect to polarization
$-K_X$.
\end{proof}

\end{document}